\documentclass[12pt]{article}

\usepackage{amsmath,amssymb}
\usepackage{mathptmx}       
\usepackage{helvet}         
\usepackage{courier}        
\usepackage{type1cm}        
\usepackage{makeidx}         
\usepackage{graphicx}        
\usepackage{multicol}        
\usepackage[bottom]{footmisc}
\usepackage{amsmath,amsfonts,amssymb}

\makeindex             

\def\Omegaext{{\Omega^c}}
\newcommand{\BLm}{{\mbox{\rm BL}^m({\mathbb R}^d)}}
\newcommand{\eremk}{\hbox{}\hfill\rule{0.8ex}{0.8ex}}

\newtheorem{theorem}{Theorem}[section]
\newtheorem{corollary}[theorem]{Corollary}

\newtheorem{lemma}[theorem]{Lemma}
\newtheorem{proposition}[theorem]{Proposition}
\newtheorem{remark}[theorem]{Remark}
\newenvironment{proof}[1][Proof]{\noindent\textbf{#1.} }{\ \rule{0.5em}{0.5em}}

\evensidemargin=\oddsidemargin
\newdimen\dummy
\dummy=\oddsidemargin
\begin{document}

\title{On thin plate spline interpolation}
\author{M. L\"ohndorf  \\
Kapsch TrafficCom, Am Europlatz 2, A-1120 Wien
\and J.M. Melenk\thanks{ melenk@tuwien.ac.at }  \\
Technische Universit\"at Wien, A-1040 Wien}
%
%
\maketitle

\abstract{We present a simple, PDE-based proof of the result \cite{johnson01} by M.~Johnson that 
the error estimates of J.~Duchon \cite{duchon78} for thin plate spline interpolation can be improved by $h^{1/2}$. 
We illustrate that ${\mathcal H}$-matrix techniques can successfully be employed to solve very large 
thin plate spline interpolation problems. 
}
\section{Introduction and Main Results}
Interpolation with so-called thin plate splines (also known as surface splines, $D^m$-splines, or polyharmonic splines) 
is a classical topic in spline theory. It is concerned with the following interpolation problem 
(\ref{eq:variational-formulation}): 
Given a (sufficiently smooth) function $f$ and points 
$x_i \in {\mathbb R}^d$, $i=1,\ldots,N$, find the minimizer $If$ of the problem
\begin{align} 
\label{eq:variational-formulation}
\mbox{ minimize } \qquad 
|v|_{H^m({\mathbb R}^d)} \qquad \mbox{ under the constraint $v(x_i) = f(x_i)$, $i=1,\ldots,N$.}
\end{align}
Here, the seminorm $|v|_{H^m({\mathbb R}^d)}$ is induced by the bilinear form 
\begin{equation}
\langle v,w\rangle_{m} := 
\sum_{|\alpha| = m} \frac{m!}{\alpha!} \int_{{\mathbb R}^d} D^\alpha v D^\alpha w\,dx.
\end{equation}
For $m > d/2$ and under very mild conditions on the point distribution, a unique minimizer $If$ exists. 
The name ``thin plate splines'' originates from the fact in the simplest case $m = d = 2$, $If$ can be represented
in terms of translates of the fundamental solution of the biharmonic equation. For general $m$ the interpolant 
$If$ can be expressed in terms fundamental solutions of $\Delta^m$:  
There are constants $c_i \in {\mathbb R}$, $i=1,\ldots,N$, 
and a polynomial $\pi \in {\mathbb P}_{m-1}$
of degree $m-1$ such that (with the Euclidean norm $\|\cdot\|_2$ on ${\mathbb R}^d$)
\begin{equation} 
\label{eq:interpolation-representation}
If(x) = \sum_{i=1}^N  c_i \phi_m(\|x - x_i\|_2) + \pi_{m-1}(x), 
\qquad \sum_{i=1}^N c_i q(x_i) = 0 \qquad \forall q \in {\mathbb P}_{m-1},
\end{equation}
where $\phi_m$ is given explicitly by 
\begin{equation}
\phi_m(r) = 
\begin{cases}
r^{2m-d} \log r & d \mbox{ even } \\
r^{2m-d} & d \mbox{ odd.} 
\end{cases}
\end{equation}
The representation (\ref{eq:interpolation-representation}) allows one to reformulate (\ref{eq:variational-formulation})
as the problem of finding the coefficients $c_i$ and the polynomial $\pi_{m-1}$ so that the (constrained) 
interpolation problem (\ref{eq:interpolation-representation}) is solved. 
The classical error analysis for (\ref{eq:variational-formulation}) 
is formulated in terms fill-distance: For a bounded domain $\Omega \subset {\mathbb R}^d$ 
and points $X_N = \{x_i\,|\, i=1,\dots,N\} \subset\Omega$, the \emph{fill distance} $h(X_N)$ is given by 
\begin{equation} 
\label{eq:fill-distance}
h(X_N):= \sup_{x \in {\Omega}} \inf_{i=1,\ldots,N} \|x - x_i\|_2. 
\end{equation} 
Starting with the seminal papers by J.~Duchon \cite{duchon76,duchon78} the error $f - If$ on $\Omega$ is controlled 
in terms of $h$ and the regularity properties of $f$ (on $\Omega$): 
\begin{proposition}[\protect{\cite[Prop.~3]{duchon78}}]
\label{prop:duchon}
Let $\Omega \subset {\mathbb R}^d$ be a bounded Lipschitz domain. Let $m > d/2$, $k \in {\mathbb N}$, 
$p \in [2,\infty]$ be such that $H^m(\Omega) \subset W^{k,p}(\Omega)$. Then, there are constants 
$h_0$, $C_1$, $C_2>0$ depending only on $\Omega$, $m$, $d$ such that for any collection 
$X_N = \{x_1,\ldots,x_N\} \subset \Omega$ with fill distance $h:=h(X_N) \leq h_0$
$$
\sum_{|\alpha| = k} \|D^\alpha (f - If)\|_{L^p(\Omega)} \leq C_1 h^{m-k-d/2+d/p} |E^\Omega f - If|_{H^m({\mathbb R}^d)} 
\leq C_2 h^{m-k-d/2+d/p} |f|_{H^{m}(\Omega)}; 
$$
here, $E^\Omega f$ denotes the minimum norm extension of $f$ defined in (\ref{eq:minimal-norm-extension}). 
\end{proposition}
In Proposition~\ref{prop:duchon} and throughout the present note, we will use the standard notation for Sobolev
spaces $W^{s,p}$ and Besov spaces $B^{s}_{2,q}$; we refer to \cite{tartar07} for their definition. Interpolation space will always be understood by the so-called ``real method'' (also known as ``$K$-method'')
as described, e.g., in  \cite{tartar07,triebel95}. 
We will use extensively that the scales of Sobolev and Besov spaces are interpolation spaces.  
We will also use 
the notation $|\nabla^j f|^2 = \sum_{|\alpha| = j} \frac{j!}{\alpha!} |D^\alpha f|^2$. 

It is worth noting that the interpolation operator $I$ is a projection so that $I (f - If) = 0$. 
Proposition~\ref{prop:duchon} applied to the function $f - If$ therefore  yields 
\begin{corollary}
\label{cor:duchon}
Under the assumptions of Proposition~\ref{prop:duchon} there holds 
$$
\sum_{|\alpha| = k} \|D^\alpha (f - If)\|_{L^p(\Omega)} 
\leq C_2 h^{m-k-d/2+d/p} |f - If|_{H^{m}(\Omega)}. 
$$
\end{corollary}

A natural question in connection with Proposition~\ref{prop:duchon} is whether the convergence rate can be 
improved by requiring additional regularity of $f$. 
It turns out that boundary effects limits this. We mention that a doubling of the convergence rate 
is possible by imposing certain homogeneous boundary conditions on high order derivatives as 
shown in \cite{melenk-gutzmer01} and, more abstractly, in \cite{schaback99}. 
If this highly fortuitous setting is not given, then only a small further gain is possible as shown 
by M.~Johnson, \cite{johnson01,johnson04}. For example, he showed that a gain of $h^{1/2}$ is possible 
if $f \in B^{m+1/2}_{2,1}(\Omega)$ 
and $\partial\Omega$ is sufficiently smooth. The purpose the present note is to give a short and simple proof 
of this result using different tools, namely, those  from elliptic PDE theory. 
The techniques also open the door to reducing the 
smoothness assumptions on $\partial\Omega$ in 
\cite{johnson01,johnson04} to Lipschitz continuity as discussed in more detail in Remark~\ref{rem:discussion}. 
Our main result therefore is a simpler proof of: 
\begin{proposition}[\protect{\cite{johnson01}}]
\label{prop:main}
Let $\Omega \subset {\mathbb R}^d$ be a bounded Lipschitz domain with sufficiently smooth boundary. Then there
are constants $h_0$, $C_1$, $C_\delta > 0$ that  depend solely on $\Omega$, $m$, $d$, and $\delta$ 
such that for any collection 
$X = \{x_1,\ldots,x_N\}\subset\Omega$ with fill distance $h:= h(X_N) \leq h_0$ there holds 
\begin{align}
\label{eq:prop:main-1}
|E^\Omega f - If|_{H^m({\mathbb R}^d)} & \leq C_1 h^{1/2} \|f\|_{B^{m+2/1}_{2,1}(\Omega)}, \\
\label{eq:prop:main-2}
|E^\Omega f - If|_{H^m({\mathbb R}^d)} & \leq C_\delta h^{\delta} \|f\|_{H^{m+\delta}(\Omega)}, 
\quad 0 \leq \delta < 1/2.
\end{align}
In particular, therefore, the estimates of \cite[Prop.~3]{duchon78} (i.e., Prop.~\ref{prop:duchon}) can 
be improved by $h^{1/2}$ for $f \in B^{m+1/2}_{2,1}(\Omega)$ and by $h^{\delta}$ 
for $f \in H^{m+\delta}(\Omega)$. 
\end{proposition}
\begin{remark} 
A common route to error estimates for $f - If$ is via the so-called ``power function'' $P(x)$. Indeed, 
classical pointwise estimates take the form 
$|f(x) - If(x)| \leq P(x) |E^\Omega f - If|_{H^m({\mathbb R}^d)}$ 
(cf., e.g.,  \cite[Prop.~{5.3}]{buhmann03}, \cite[Thm.~{11.4}]{wendland05}) and $P$  is subsequently 
estimated in terms of the fill distance $h$.  Thus, Proposition~\ref{prop:main}
allows for improving estimates in this setting. 
\eremk
\end{remark}
We close this section by referring the reader to the monographs \cite{wendland05,buhmann03} as well
as \cite{johnson04a} for further details
on the approximation properties of radial basis functions, in particular, thin plate splines.
\section{Proof of Proposition~\ref{prop:main}}
\subsection{Tools}
The precise formulation of the minimization problem (\ref{eq:variational-formulation}) is based on the 
classical \emph{Beppo-Levi space} $\BLm$, which is defined as  
$$
\BLm:= \{u \in {\mathcal D}^\prime\,|\, \nabla^m u \in L^2({\mathbb R}^d)\}. 
$$
We refer to \cite{deny-lions54} and \cite[Sec.~{10.5}]{wendland05} for more properties of the space $\BLm$; 
in particular, 
$C^\infty_0({\mathbb R}^d)$ is dense in $\BLm$ (see \cite[Thm.~{10.40}]{wendland05} for the precise notion). 
We also need the \emph{minimum norm extension} $E^\Omega:H^m(\Omega) \rightarrow \BLm$ given by 
\begin{equation} 
\label{eq:minimal-norm-extension} 
E^\Omega U = \arg\min\{|u|_{H^m({\mathbb R}^d)}\,|\, u \in \BLm, \quad u|_{\Omega} = U\}. 
\end{equation} 
The minimization property in (\ref{eq:minimal-norm-extension}) implies the orthogonality 
\begin{equation}
\langle E^\Omega U, v\rangle_m = 0 \qquad \forall v \in \{v \in \BLm\,|\, v|_{\Omega} =0\}.
\end{equation}
The connection with elliptic PDE theory arises from the fact that 
$E^\Omega U$ satisfies an elliptic  PDE in $\Omegaext:={\mathbb R}^d \setminus \overline{\Omega}$:  
\begin{equation}
\Delta^m E^\Omega U = 0 \qquad \mbox{ in $\Omegaext$}. 
\end{equation}
It will be convenient to decompose $B(u,v):= \langle u,v\rangle_m= 
\sum_{|\alpha| = m} \frac{m!}{\alpha!} \int_{{\mathbb R}^d} D^\alpha u D^\alpha v 
$
as 
$
B(u,v)= B_\Omega(u,v) + B_{\Omegaext}(u,v)$, 
where  
$$
B_\Omega(u,v):= \sum_{|\alpha| = m} \frac{m!}{\alpha!} \int_\Omega D^\alpha u D^\alpha  v, 
\quad 
B_{\Omegaext}(u,v):= \sum_{|\alpha| = m} \frac{m!}{\alpha!} \int_{\Omegaext} D^\alpha u D^\alpha  v. 
$$
The trace mapping is continuous 
$H^{1/2+\varepsilon}(\Omega) \rightarrow H^{\varepsilon}(\partial\Omega)$ for $\varepsilon \in (0,1/2]$; 
however, the limiting case $\varepsilon = 0$ is not true; it is true if the Sobolev
space $H^{1/2}(\Omega)$ is replaced with the slightly smaller Besov space $B^{1/2}_{2,1}(\Omega)$: 
\begin{lemma}[Trace theorem]
\label{lemma:Besov-trace}
Let $\Omega \subset {\mathbb R}^d$ be a Lipschitz domain, $k \in{\mathbb N}_0$. Then there exists
$C > 0$ such that the multiplicative estimate 
$
\|u\|^2_{L^2(\partial\Omega)} \leq C \|u\|_{L^2(\Omega)} \|u\|_{H^1(\Omega)}$
holds as well as 
\begin{equation}
\label{eq:lemma:Besov-trace-10}
\|u\|_{L^2(\partial\Omega)} \leq C \|u\|_{B^{1/2}_{2,1}(\Omega)}, 
\qquad 
\|\nabla^k u\|_{L^2(\partial\Omega)} \leq C \|u\|_{B^{k+1/2}_{2,1}(\Omega)}. 
\end{equation}
\end{lemma}
\begin{proof}
The case $k \ge 1$ in (\ref{eq:lemma:Besov-trace-10}) follows immediately from the case $k = 0$. The case $k = 0$ is 
discussed in \cite[Thm.~{2.9.3}]{triebel95} for the case of a half-space. The generalization
to Lipschitz domains can be found, for example, in \cite[Lemma~{1.10}]{adolphsson-pipher98}. 
\end{proof}
\subsection{An interpolation argument}  
The following technical result, which is of independent interest, will be used to reduce regularity assumptions
to $B^{m+1/2}_{2,1}(\Omega)$. 
\begin{lemma}
\label{lemma:functional-interpolation}
Let $X_1 \subset X_0$ be two Banach spaces with continuous embedding. Let  $q \in [1,\infty]$, $\theta \in (0,1)$. Define (by the real method
of interpolation) $X_\theta:= (X_0,X_1)_{\theta,q}$  for $\theta \in (0,1)$. 
Let $0 < \theta_1 < \theta_2 < \cdots < \theta_n < 1$ be fixed and assume that 
$l \in X_0^\prime$ satisfies for some $C_0$, $C_1$, $\varepsilon > 0$ 
\begin{eqnarray*}
|l(f)| &\leq & C_0 \|f\|_{X_0} \qquad \forall f \in X_0, \\
|l(f)| &\leq & C_1 \left[ \sum_{i=1}^n \varepsilon^{\theta_i} \|f\|_{X_{\theta_i}} + \varepsilon \|f\|_{X_1}\right] 
\qquad \forall f \in X_1. 
\end{eqnarray*}
Then there exists a constant $C > 0$ that is independent of $\varepsilon$ such that 
$$
|l(f)| \leq C \varepsilon^{\theta_1} \|f\|_{X_{\theta_1}} \qquad \forall f \in X_{\theta_1}. 
$$
\end{lemma}
\begin{proof}
We start with the special case $n = 1$ and we abbreviate $\theta = \theta_1$. 
Let $f \in X_\theta$. 
By definition of the $K$-functional we may choose $\widetilde f \in X_1$ with 
\begin{equation}
\label{eq:lemma:functional-interpolation-1}
\|f - \widetilde f\|_{X_0} + \varepsilon \|\widetilde f\|_{X_1} \leq 2 K(\varepsilon,f). 
\end{equation}
Using the linearity  of $l$, we can bound 
\begin{eqnarray*}
|l(f)|  &=& |l(f - \widetilde f) + l(\widetilde f)| \leq 
C_0 \|f - \widetilde f\|_{X_0} + 
C_1 \left[ \varepsilon^\theta \|\widetilde f\|_{X_\theta} + \varepsilon \|\widetilde f\|_{X_1} \right]
\stackrel{ (\ref{eq:lemma:functional-interpolation-1})}{ \leq } 
C K(\varepsilon,f) + \varepsilon^\theta \|\widetilde f\|_{X_\theta} \\
&\leq& C K(\varepsilon,f) + \varepsilon^\theta \|f - \widetilde f\|_{X_\theta} + 
\varepsilon^\theta\|f\|_{X_\theta}. 
\end{eqnarray*}
We now use the bound $\|f - \widetilde f\|_{X_\theta} \leq 3 K(\varepsilon,f)$ from
\cite[eqn.~(2.8)]{bramble-scott78} and then 
$K(\varepsilon,f) \leq C \varepsilon^\theta \|f\|_{X_\theta}$ 
(see, e.g., \cite[Thm.~{1.3.3}]{triebel95}) to conclude  
$$
|l(f)| \leq C \varepsilon^\theta \|f\|_{X_\theta}. 
$$
We now consider the general case $n > 1$. We choose $\widetilde f$ as in 
(\ref{eq:lemma:functional-interpolation-1}) and proceed as above to get 
\begin{align}
\label{eq:lemma:functional-interpolation-2}
|l(f)|  & = |l(f - \widetilde f) + l(\widetilde f)| \leq 
C_0 \|f - \widetilde f\|_{X_0} + C_1\!\! \left[ \varepsilon^{\theta_1} \|\widetilde f\|_{X_{\theta_1}} + 
 \sum_{i=2}^n \varepsilon^{\theta_i} \|\widetilde f\|_{X_{\theta_i}} + 
 \varepsilon\|\widetilde f\|_{X_1}\right].
\end{align}
In order to treat the terms involving $\|\widetilde f\|_{X_{\theta_i}}$ for $i \ge 2$, we 
use the reiteration theorem to infer $X_{\theta_i} = (X_{\theta_1},X_{1})_{s_i,q}$, where 
$s_i \in (0,1)$ is given by 
$$
\theta_i = \theta_1 (1 - s_i) + s_i.  
$$
Next, the interpolation inequality 
$\|\widetilde f\|_{X_{\theta_i}}
\leq C\|\widetilde f\|_{X_{\theta_1}}^{1-s_i}\|\widetilde f\|_{X_1}^{s_i}$ 
together with the elementary bound $ab \leq a^p + b^q$ ($a$, $b > 0$, $p$, $q > 1$ with $1/p+1/q = 1$)
gives 
\begin{align*}
\varepsilon^{\theta_i} \|\widetilde f\|_{X_{\theta_i}} & \leq C 
\varepsilon^{\theta_i  - s_i} \|\widetilde f\|_{X_{\theta_1}}^{1-s_i} \, \varepsilon^{s_i}\|\widetilde f\|_{X_1}^{s_i} 
\leq C \left[ 
\varepsilon^{(\theta_i - s_i)/(1 - s_i)} \|\widetilde f\|_{X_{\theta_1}} + 
\varepsilon \|\widetilde f\|_{X_1} 
\right] \\
& = C \left[ 
\varepsilon^{\theta_1} \|\widetilde f\|_{X_{\theta_1}} + 
\varepsilon \|\widetilde f\|_{X_1} 
\right].
\end{align*}
Inserting this result in (\ref{eq:lemma:functional-interpolation-2}), we get together with 
(\ref{eq:lemma:functional-interpolation-1})
$$
|l(f)| \leq C \left[ K(\varepsilon,f) + \varepsilon^{\theta_1} \|\widetilde f\|_{X_{\theta_1}}\right]. 
$$
Reasoning as in the case $n=1$ now allows us to conclude the argument. 
\end{proof}
\subsection{Elliptic regularity}
\begin{lemma}
\label{lemma:ADN}
Let $\Omega \subset {\mathbb R}^d$ be a bounded Lipschitz domain with a \emph{smooth} boundary. 
Let $m \in {\mathbb N}$ and $k \in {\mathbb N}_0$. Then there is $C_{\Omega,m,k}$ depending only
on $\Omega$, $m$, $k$ such that the following is true: If $g \in H^{-m+k}(\Omega)$ and $u$ is the (variational)
solution of the Dirichlet problem 
$$
\Delta^m u = g 
\quad \mbox{ in $\Omega$}, 
\qquad u = \partial_n u = \cdots \partial_n^{m-1} u = 0 
\quad \mbox{ on $\partial\Omega$}, 
$$
then $u \in H^{m+k}(\Omega)$ with the {\sl a priori} bound 
$$
\|u\|_{H^{m+k}(\Omega)} \leq C_{\Omega,m,k} \|g\|_{H^{-m+k}(\Omega)}. 
$$
\end{lemma}
\begin{proof}
This regularity result is a special case of a more general result for the regularity of 
solutions of elliptic systems, \cite{adn59,adn64}. Self-contained proofs of this result can 
also be found, for example, in \cite[Sec.~{20}]{wloka82} and in 
\cite[Chap.~2, Thm.~{8.2}]{lions-magenes72}. 
\end{proof}
The minimum norm extension $E^\Omega:H^m(\Omega)\rightarrow \BLm$ satisfies 
\begin{equation}
\label{eq:minimum-norm-extension-bound}
|E^\Omega f|_{H^m({\mathbb R}^d)} \leq C_\Omega \|f\|_{H^m(\Omega)}. 
\end{equation}
However, for smooth $\partial\Omega$, it has additional mapping properties: 
\begin{corollary}
\label{cor:shift-theorem}
Let $\Omega$ be a bounded Lipschitz domain with a smooth boundary and let $\overline{\Omega}$ be contained in the (open) ball $B_R(0)$ 
of radius $R$ centered at $0$. 
For each $j \in \{0,\ldots,m\}$ there is a constant $C_{j,\Omega} > 0$ depending only
on $j$, $\Omega$, and $R$ such that the following is true for the minimum norm extension 
$E^\Omega: H^m(\Omega) \rightarrow \BLm$: It is also  
a bounded linear map $H^{m+j}(\Omega) \rightarrow H^{m+j}(B_R(0)\setminus\overline{\Omega})$ and, 
with $\gamma_0^c$ denoting the trace operator for $B_R(0)\setminus\overline{\Omega}$,
\begin{align}
\label{eq:cor:shift-theorem-10}
\|\gamma_0^c(\nabla^{m+j} E^\Omega f)\|_{L^2(\partial\Omega)} \leq 
C_{j,\Omega} \|f\|_{B^{m+j+1/2}_{2,1}(\Omega)},
\end{align}
\end{corollary}
\begin{proof}
We write $\widetilde \Omega:= B_R(0)\setminus\overline{\Omega}$. The operator $E^\Omega$ is clearly
a bounded linear map $E^\Omega:H^m(\Omega) \rightarrow H^m(\widetilde \Omega)$. From
Lemma~\ref{lemma:ADN}, we also see that $E^\Omega$ maps $H^{2m}(\Omega)$ boundedly into
$H^{2m}(\widetilde \Omega)$: We denote by $E$ the universal extension operator 
of \cite[Chap.~{VI}, 3]{stein70}, which we may choose such that $\operatorname*{supp} Ef \subset B_R(0)$. Next, 
we write $E^\Omega f$ in the form $E^\Omega f = E f + u$, where $Ef \in H^{2m}(\widetilde\Omega)$ 
(since $f \in H^{2m}(\Omega)$) and $u$ solves the differential equation 
$$
\Delta^m u = -\Delta^m Ef \in L^2(\widetilde\Omega) \quad \mbox{ in $\widetilde\Omega$}, 
\qquad u = \partial_n u = \cdots = \partial_n^{m-1} u = 0 \qquad \mbox{ on $\partial\widetilde \Omega$}.
$$
Lemma~\ref{lemma:ADN} then gives $u \in H^{2m}(\widetilde\Omega)$ with the 
{\sl a priori} estimate $\|u\|_{H^{2m}(\widetilde\Omega)} \leq C \|\Delta^m Ef\|_{L^2(\widetilde\Omega)}
\leq C \|Ef\|_{H^{2m}(\widetilde\Omega)} \leq C \|f\|_{H^{2m}(\Omega)}$. We have thus obtained
\begin{equation}
\label{eq:cor:shift-theorem-1}
\|E^\Omega f\|_{H^m(\widetilde\Omega)} \leq C \|f\|_{H^m(\Omega)}, 
\qquad 
\|E^\Omega f\|_{H^{2m}(\widetilde\Omega)} \leq C \|f\|_{H^{2m}(\Omega)}.
\end{equation}
An interpolation argument then gives us 
$$
\|E^\Omega f\|_{B^{m+1/2+j}_{2,1}(\widetilde\Omega)} \leq C \|f\|_{B^{m+j+1/2}_{2,1}(\Omega)}, 
\qquad j=0,\ldots,m-1. 
$$
By the trace theorem (Lemma~\ref{lemma:Besov-trace}), we arrive at  
$\|\nabla^{j+m} E^\Omega f\|_{L^2(\partial\Omega)} \leq C \|f\|_{B^{m+j+1/2}_{2,1}(\Omega)}$
for $j=0,\ldots,m-1$. 
\end{proof}
\subsection{PDE-based proof of Proposition~\ref{prop:main}}
\begin{lemma} 
\label{lemma:norm-EOmega}
Let $\Omega$ be a Lipschitz domain. Then 
$$
|E^\Omega f - I f|_{m} \leq C_\Omega |f - If|_{H^m(\Omega)}. 
$$
\end{lemma}
\begin{proof} 
We exploit that $\Delta^m ( E^\Omega f - I f) = 0$ in ${\Omegaext}$. To that end, let again $E$ be the 
universal extension of operator 
of \cite[Chap.~{VI}, 3]{stein70}. We write $E^\Omega f - I f = E (f - If) + \delta$ for some 
$\delta \in \BLm$ with $\delta|_{\Omega} = 0$. We get
\begin{align*}
|E^\Omega f - If|^2_m & = B_{\Omega}(f - If, f - If) + B_{\Omegaext}(E^\Omega f - If, E(f - If) + \delta) \\
 & = |f - If|^2_{H^m(\Omega)}  + B_{\Omegaext}(E^\Omega f - If, E(f - If)),  
\end{align*}
where we used integration by parts and that $\delta|_{\Omega} \equiv 0$; 
the integration by parts does not produce any terms ``at infinity'' since $C^\infty_0({\mathbb R}^d)$ is 
dense in $\BLm$ (in the sense described in \cite[Thm.~{10.40}]{wendland05}) and thus $\delta$ 
can be approximated by such compactly supported functions. The continuity of $E$ implies 
$$
|E^\Omega f - If |_m \leq C_\Omega \|f - If\|_{H^m(\Omega)}, 
$$
and the reduction  to a seminorm follows from the Deny-Lions Lemma and fact that $I$ reproduces 
polynomials of degree $m-1$.  
\end{proof}
The solution $If$ of the minimization problem (\ref{eq:variational-formulation}) satisfies the
orthogonality condition 
\begin{equation}
\label{eq:orthogonality}
\langle E^\Omega f - If,If \rangle_{m} = 0 
\end{equation}
since $E^\Omega f - If \in \BLm$ and $(E^\Omega f - If)(x_i) = f(x_i) - If(x_i) = 0$, $i =1,\ldots,N$. 
Therefore, 
\begin{align}
\nonumber 
\langle E^\Omega f - If,E^\Omega f - If\rangle_m &= 
\langle E^\Omega f - If,E^\Omega f \rangle_m \\ 
\label{eq:nuce-0}
& = B_\Omega(f - If,f) + 
B_{\Omegaext}(E^\Omega f - If,E^\Omega f). 
\end{align}
These last two terms are treated separately in Lemmas~\ref{lemma:domain-part}, \ref{lemma:exterior-domain-part}. 
Inserting (\ref{eq:lemma:domain-part-10}),  (\ref{eq:lemma:exterior-domain-part-10}) 
in (\ref{eq:nuce-0}) we get 
$$
|E^\Omega f - If|^2_{H^m({\mathbb R}^d)} \leq C h^{1/2} \|f\|_{B^{m+1/2}_{2,1}(\Omega)} |f - If|_{H^m(\Omega)}, 
$$
which readily implies (\ref{eq:prop:main-1}) of
Proposition~\ref{prop:main}. The bound (\ref{eq:prop:main-2}) follows from 
(\ref{eq:prop:main-1}) and an interpolation argument 
since the reiteration theorem asserts 
for $0 < \delta < 1/2$ that 
$H^{m+\delta}(\Omega) = (H^m(\Omega),B^{m+1/2}_{2,1}(\Omega))_{2 \delta,2}$
and $|E^\Omega f - If|_{H^m({\mathbb R}^d)} \leq C \|f\|_{H^m(\Omega)}$, which follows from 
combining (\ref{eq:orthogonality}) and (\ref{eq:minimum-norm-extension-bound}). 
\begin{lemma} 
\label{lemma:domain-part}
Let $\Omega$ be a Lipschitz domain. Then: 
\begin{align}
\label{eq:lemma:domain-part-10}
|B_\Omega(f - If,f)| \leq C_\Omega h^{1/2} |f - If|_{H^m(\Omega)} \|f\|_{B^{m+1/2}_{2,1}(\Omega)}. 
\end{align}
\end{lemma}
\begin{proof}
Let $\widetilde f \in H^{m+1}(\Omega)$. Integration by parts once gives 
\begin{align}
\label{eq:nuce-1}
 \Bigl| B_\Omega(&f - If,\widetilde f) \Bigr|  \lesssim  \\
\nonumber 
& 
\|\nabla^{m-1}(f - If)\|_{L^2(\partial\Omega)} 
\|\nabla^{m}\widetilde f\|_{L^2(\partial\Omega)} + 
\|\nabla^{m-1} (f - If)\|_{L^2(\Omega)} \|\nabla^{m+1}\widetilde f\|_{L^2(\Omega)}. 
\end{align}
The multiplicative trace inequality 
$\|z\|^2_{L^2(\partial\Omega)} \lesssim \|z\|_{L^2(\Omega)} \|z\|_{H^1(\Omega)}$, 
Corollary~\ref{cor:duchon} with $k = m-1$, and the trace estimate 
$\|\nabla^m z\|_{L^2(\partial\Omega)} \lesssim \|z\|_{B^{m+1/2}_{2,1}(\Omega)}$ 
yield 
\begin{align*}
& \left| B_\Omega(f - If,\widetilde f) \right|  \lesssim   \\
&
\left[ \|\nabla^{m-1}(f - If)\|^{1/2}_{L^2(\Omega)} \|f - If\|^{1/2}_{H^m(\Omega)} 
         \right]\|\nabla^m \widetilde f\|_{L^2(\partial\Omega)} + 
\|\nabla^{m-1} (f - If)\|_{L^2(\Omega)} \|\nabla^{m+1}\widetilde f\|_{L^2(\Omega)} \\
&\lesssim \left[ 
h^{1/2} |f - If|_{H^m(\Omega)} \|\nabla^m \widetilde f\|_{L^2(\partial\Omega)} + 
h |f - If|_{H^m(\Omega)} 
\|\nabla^{m+1} \widetilde f\|_{L^2(\Omega)} \right]\\
&\lesssim
\left[h^{1/2} \|\widetilde f\|_{B^{m+1/2}_{2,1}(\Omega)} + h \|\widetilde f\|_{H^{m+1}(\Omega)}\right]
|f - If|_{H^m(\Omega)}. 
\end{align*}
We conclude that the linear functional $\widetilde f \mapsto B_\Omega(f - If,\widetilde f)$ 
satisfies 
\begin{eqnarray*}
|B_\Omega( f- If,\widetilde f)| &\leq& C |f - If|_{H^m(\Omega)} \|\widetilde f\|_{H^m(\Omega)}, \\
|B_\Omega( f- If,\widetilde f)| &\leq& C |f - If|_{H^m(\Omega)} 
\left[ h^{1/2} \|\widetilde f\|_{B^{m+1/2}_{2,1}(\Omega)} + h \|\widetilde f\|_{H^{m+1}(\Omega)}\right]; 
\end{eqnarray*}
since $B^{m+1/2}_{2,1}(\Omega) = (H^{m}(\Omega),H^{m+1}(\Omega))_{1/2,1}$ 
Lemma~\ref{lemma:functional-interpolation} implies the 
estimate (\ref{eq:lemma:domain-part-10}). 
\end{proof}
We now turn to the second part of (\ref{eq:nuce-1}). The key step is to observe that the 
minimum norm extension $E^\Omega f$ satisfies the homogeneous differential equation 
$\Delta^m E^\Omega f = 0$ in $\Omegaext$. 
\begin{lemma}
\label{lemma:exterior-domain-part} 
Let $\Omega$ be a bounded Lipschitz domain with a sufficiently smooth boundary. Then:
\begin{align}
\label{eq:lemma:exterior-domain-part-10} 
\left| B_{\Omegaext}(E^\Omega f - If,E^\Omega f)\right| \leq C_\Omega h^{1/2} |f - If|_{H^m(\Omega)} \|f\|_{B^{m+1/2}_{2,1}(\Omega)} . 
\end{align}
\end{lemma}
\begin{proof} 
Let $\widetilde f \in H^{2m}(\Omega)$. By Corollary~\ref{cor:shift-theorem}, we have 
$E^\Omega \widetilde f \in H^{2m}(B_R(0)\cap \Omegaext)$ for every $R>0$ sufficiently large. 
Furthermore, $\Delta^m E^\Omega \widetilde f  = 0$ in $\Omegaext$. Next, $m$-fold integration by parts yields 
\begin{equation}
\label{eq:nuce-3} 
\left| B_{\Omegaext}( E^\Omega f - If,E^\Omega \widetilde f) \right| 
\lesssim  \sum_{j=1}^{m} \|\nabla^{m-j}(E^\Omega f - If)\|_{L^2(\partial\Omega)} 
                       \|\gamma_0^c \nabla^{m+j-1} E^\Omega \widetilde f \|_{L^2(\partial\Omega)}. 
\end{equation}
The integration by parts does not produce any terms ``at infinity'' since $C^\infty_0({\mathbb R}^d)$ is 
dense in $\BLm$ (in the sense described in \cite[Thm.~{10.40}]{wendland05}) and thus 
$E^\Omega f - If \in \BLm$ can be approximated by such compactly supported functions. 

Since $\nabla^{j} E^\Omega f = \nabla^j f$ on $\partial\Omega$ for $j=0,\ldots,m-1$, we 
use again the multiplicative trace inequality and Corollary~\ref{cor:duchon} to get 
\begin{align}
\nonumber 
\left| B_{\Omegaext}( E^\Omega f - If,E^\Omega \widetilde f) \right| 
& \leq 
C |f - If|_{H^m(\Omega)} \sum_{j=1}^m h^{-1/2+j} \|\gamma_0^c \nabla^{m+j-1} E^\Omega \widetilde f\|_{L^2(\partial\Omega)} \\
\label{eq:nuce-4} 
& \stackrel{(\ref{eq:cor:shift-theorem-10})}{\leq}
C |f - If|_{H^m(\Omega)} \sum_{j=1}^m h^{-1/2+j} \|\widetilde f\|_{B^{m+j-1/2}_{2,1}(\Omega)}.  
\end{align}
We reduce the regularity requirement on $\widetilde f$ by applying 
Lemma~\ref{lemma:functional-interpolation} to 
$\widetilde f \mapsto B_{\Omegaext}(E^\Omega f - If,E^\Omega \widetilde f)$:  We observe that 
the reiteration theorem of interpolation allows us to identify
$$
B^{m+j-1/2}_{2,1}(\Omega) = (H^m(\Omega),B^{2m-1/2}_{2,1}(\Omega))_{\theta_j,1}, 
\qquad \theta_j = \frac{j-1/2}{m-1/2}; 
$$
hence, 
we get (\ref{eq:lemma:exterior-domain-part-10}) 
from an application of Lemma~\ref{lemma:functional-interpolation} with $X_0 = H^m(\Omega)$, 
$X_1 = B^{2m-1/2}_{2,1}(\Omega)$ and $\varepsilon = h^{m-1/2}$ 
since we have additionally
the stability bound 
$|B_{\Omegaext}(E^\Omega f - If,E^\Omega \widetilde f)|\leq 
C |f - If|_{H^m(\Omega)} \|\widetilde f\|_{H^m(\Omega)}$ by 
Lemma~\ref{lemma:norm-EOmega}  and 
(\ref{eq:cor:shift-theorem-1}). 
\end{proof}
\begin{remark}[Generalization to Lipschitz domains]
\label{rem:discussion} 
The proof Proposition~\ref{prop:main} relies on three ingredients: a) integration by parts arguments to treat
$B_\Omega$, 
b) the approximation properties given in \cite{duchon78} of the thin plate spline interpolation operator $I$, 
and c) regularity properties of $u:= E^\Omega f$. 
Ingredients a) and b) are already formulated for Lipschitz domains. 
However, the regularity properties of $u = E^\Omega f$ are delicate in their generalization to the case of Lipschitz domains. 
We note that $u$ solves in $\Omegaext$ the Dirichlet problem 
$$
-\Delta^m u = 0 \quad \mbox{ in $\Omegaext$},
\qquad \partial_n^{j-1} u|_{\partial\Omega} = \partial_n^{j-1} f|_{\partial\Omega}, 
\qquad j=1,\ldots,m-1. 
$$
and \cite{verchota90,pipher-verchota95,dahlberg-kenig-pipher-verchota97} 
show a shift theorem by $1/2$ in the sense that 
for $f \in B^{m+1/2}(\partial\Omega)$, one can control $\nabla^j u|_{\partial\Omega}$ for $j=0,\ldots,m$. 
This together with careful integration by parts arguments for the treatment of $B_\Omegaext$ allow for an
extension of the proof of Proposition~\ref{prop:main} to Lipschitz domain and will be given in \cite{melenk17}. 
\eremk
\end{remark}
\section{Numerical example}
\label{sec:numerics}
We illustrate Proposition~\ref{prop:main} for the case $m = d = 2$, i.e., the classical
thin plate splines.  
We employ uniformly distributed
nodes on two geometries, the unit square $\Omega_1 = (0,1)^2$ and the L-shaped domain 
$\Omega_2 = (-1/2,1/2)^2\setminus [0,1/2]^2$. As usual, we denote $r: x\mapsto \|x\|_2$. 
We interpolate 4 functions with different characters: 
the functions $r^{1.05}$ and 
$r^{2.76}$, which are, for any $\varepsilon >0$, in $H^{2.05-\varepsilon}$ and $H^{3.76-\varepsilon}$, respectively, 
and the smooth functions 
$\exp(xy)$ and $F(x,y)$, where the so-called Franke function $F$ is given by 
\begin{align*} 
F(x,y) =&  0.75 \exp(-0.25 ((9x-2)^2+(9y-2)^2) + 
0.75 \exp(- (9 x+1)^2/49 - 0.1(9y + 1)^2)+ \\
& 0.5  \exp(- 0.25((9x-7)^2  +  (9y-3)^2)
- 0.2\exp(-(9x-4)^2- (9y-7)^2). 
\end{align*}
The results are presented in Fig.~\ref{fig:results} and corroborate the assertions of 
Proposition~\ref{prop:main}, which read, for $m=2$, 
$\|f - If\|_{L^\infty(\Omega)} \leq C h^{1+\delta} \|f\|_{H^{2+\delta}(\Omega)}$ with $\delta \in [0,1/2)$ and 
$\|f - If\|_{L^\infty(\Omega)} \leq C h^{3/2} \|f\|_{B^{5/2}_{2,1}(\Omega)}$. 
These numerical results were first presented at the conference \cite{bonn-conference05}. 
\begin{figure}
\includegraphics[width=0.49\textwidth]{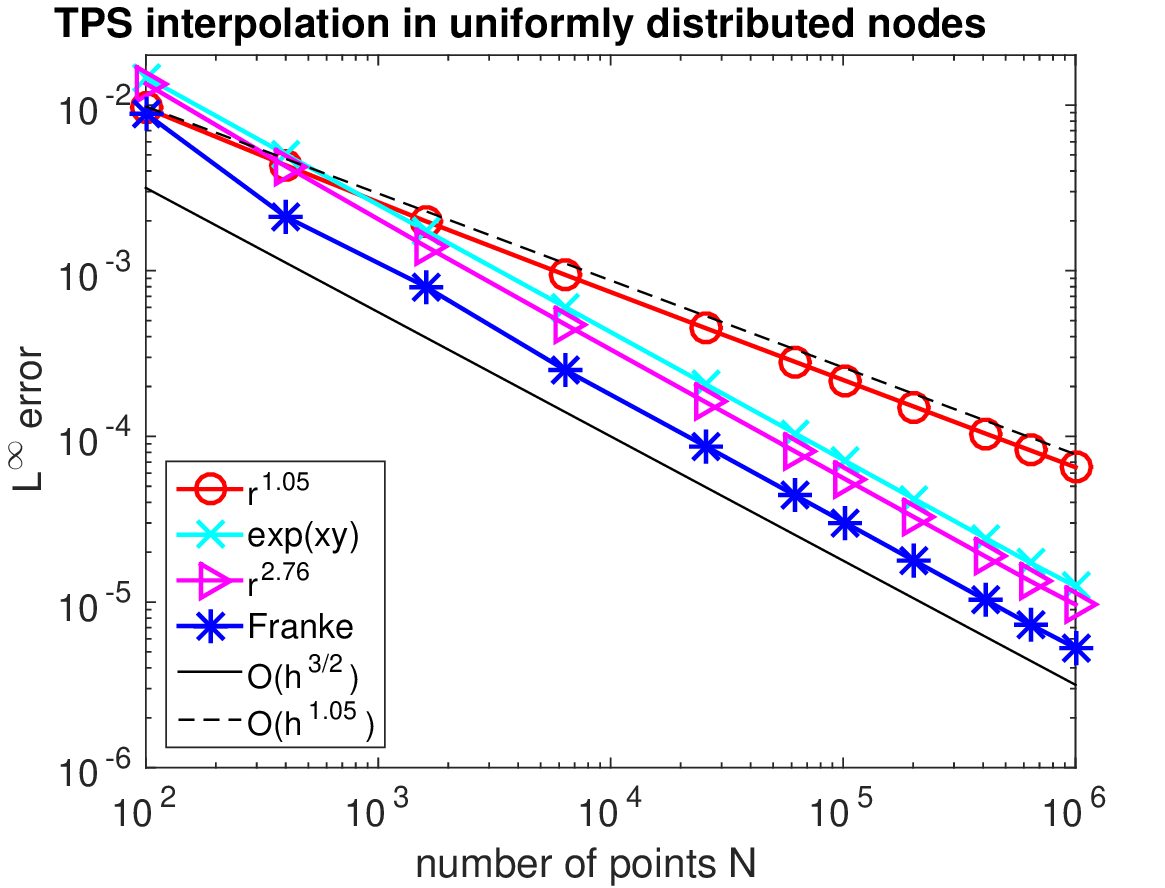}
\includegraphics[width=0.49\textwidth]{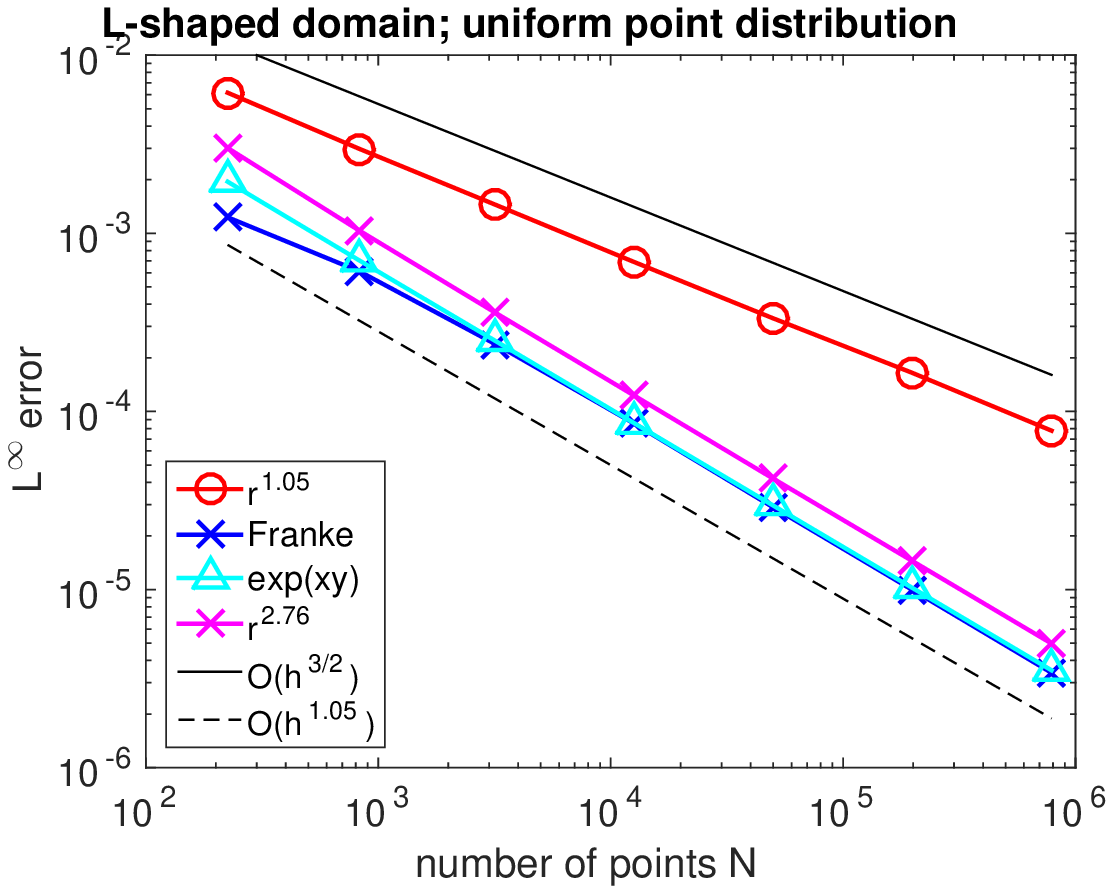}
\caption{\label{fig:results} Convergence of TPS interpolation. Left: square $\Omega_1$. Right: L-shaped domain $\Omega_2$.}
\end{figure}
\subsection{${\mathcal H}$-matrix techniques for solving the TPS interpolation problem}
The numerical solution of the thin plate interpolation
problem is numerically challenging since the system matrix is fully populated. Nevertheless, several
approaches for fast solution techniques exist. 
For example, the matrix-vector multiplication can be realized in log-linear complexity using techniques from 
fast multipole methods. This leads to efficient solution strategies based on Krylov subspace methods 
provided suitable preconditioners are available. We refer to \cite[Sec.~{15}]{wendland05}, 
\cite[Sec.~{7.3}]{buhmann03} as starting points for a literature discussion. For our calculations, 
we employed related techniques based on the concept of ${\mathcal H}$-matrices, \cite{hackbusch99,hackbusch15}. 
${\mathcal H}$-matrices come with an (approximate) factorization that can either be used as a solver 
(if the approximation is sufficiently accurate) 
or as a preconditioner in an iterative environment. The latter use has been advocated, in a different
context, for example, in \cite{bebendorf05,grasedyck05}. 

For the case $m = 2 = d$, the interpolation problem (\ref{eq:interpolation-representation}) results
in a linear system of equations of the form 
\begin{equation} 
\label{eq:TPS-matrix}
\left(\begin{array}{cc}{\mathbf P}^\top & 0 \\ {\mathbf G} & {\mathbf P} 
      \end{array}
\right) 
\left(\begin{array}{c} {\mathbf c} \\ \lambda 
      \end{array}
\right) = 
\left(\begin{array}{c} 0 \\ {\mathbf f}
      \end{array}
\right), 
\qquad {\mathbf G}_{ij} = \phi_2(\|x_i - x_j\|_2), \quad i,j=1,\ldots,N. 
\end{equation}
The matrix ${\mathbf P}^{N \times 3}$ is obtained by selecting a basis $\{b_1,b_2,b_3\}$ of 
${\mathbb P}_{1}$ (e.g., $\{1, x,y\}$) and setting ${\mathbf P}_{i,j} = b_j(x_i)$. 
The vector ${\mathbf f} \in {\mathbb R}^N$ collects
the values $f(x_i)$, the vector ${\mathbf c} \in {\mathbb R}^N$ the sought coefficients $c_i$, and the 
vector $\lambda \in {\mathbb R}^3$ is the  Lagrange multiplier for the constrained problem (\ref{eq:interpolation-representation}). The function $\phi_2(z) = z^2 \log z$ is smooth for $z > 0$. 
Lemma~\ref{lemma:chebyshev-interpolation} below shows that the function 
$(x,y) \mapsto \phi_2(\|x - y\|_2)$ can be approximated by a polynomial, which is in particular 
a \emph{separable} function, i.e. a short sum of products of functions of $x$ and $y$, only. This in turn 
implies that the fully populated matrix ${\mathbf G}$ can in fact be approximated as a blockwise low-rank matrix, 
in particular in the form of an ${\mathcal H}$-matrix, \cite{hackbusch99,hackbusch15}. 

By forming a Schur complement, the linear system of (\ref{eq:TPS-matrix}) can be transformed to 
SPD form. To that end, we select three points and rearrange the problem (\ref{eq:TPS-matrix}) as 
$$
\left(
\begin{array}{ccc} P_1^\top & 0 & P_2^\top \\
                   G_{11} & P_1 & G_{12} \\
                   G_{21} & P_2 & G_{22} 
\end{array}
\right) 
\left( 
\begin{array}{c}
{\mathbf c}_1 \\ \lambda \\ {\mathbf c}_2
\end{array}
\right)
 = 
\left( 
\begin{array}{c}
0 \\ {\mathbf f}_1 \\ {\mathbf f}_2 
\end{array}
\right)
\qquad G_{11} \in {\mathbb R}^{3 \times 3}, \quad 
G_{22} \in {\mathbb R}^{(N-3) \times (N-3)}, 
$$
where the vectors ${\mathbf c}_1$, ${\mathbf f}_1 \in {\mathbb R}^3$ and 
${\mathbf c}_2$, ${\mathbf f}_2 \in {\mathbb R}^{N-3}$ result from the permutations. The Schur complement 
$$
{\mathbf S}:= G_{22} - \left(\begin{array}{cc} G_{21} & P_2\end{array}\right)
                       \left(\begin{array}{cc} P_1^\top & 0 \\ G_{11} & P_1\end{array}\right)^{-1} 
                       \left(\begin{array}{c} P_2^\top \\ G_{12} \end{array}\right)
$$
is SPD. 
We computed an (approximate) Cholesky factorization of ${\mathbf S}$ using the library HLib \cite{Hlib}. 
This factorization can be employed as a preconditioner for a CG iteration. The ${\mathcal H}$-matrix
structure of ${\mathbf S}$ was ensured by so-called geometric clustering of the 
interpolation points. Specifically, we used this hierarchical structure to set up $G_{22}$ by 
approximating its entries with the Chebyshev interpolant as described in 
Lemma~\ref{lemma:chebyshev-interpolation}. 
In the interest of efficiency, the thus obtained ${\mathcal H}$-matrix approximation of $G_{22}$ was 
further modified by using SVD-based compression of blocks as well as coarsing 
of the block structure (these tools are provided by HLib). The matrix ${\mathbf S}$ 
is a rank-$3$ update of the matrix $G_{22}$, which can also be realized in HLib.
\begin{lemma}
\label{lemma:chebyshev-interpolation}
Let $\eta > 0$ be given. For any (closed) axiparallel boxes 
$\sigma$, $\tau \subset {\mathbb R}^2$ and a polynomial degree $p \in {\mathbb N}_0$ 
denote by $I^{Cheb}_p: C(\sigma \times \tau) \rightarrow {\mathbb Q}_p$
the tensor product Chebyshev interpolation operator associated with $\sigma \times \tau$.  Then there 
are constants $C$, $b > 0$ depending only on $\eta$ such that under the condition 
$\max\{ \operatorname*{diam}(\sigma),\operatorname*{diam}(\tau)\} \leq \eta \operatorname*{dist}(\sigma,\tau)$ there
holds 
$$
\sup_{(x,y) \in \sigma\times \tau} |\phi_2(\|x - y\|_2) - I^{Cheb}_{p} \phi_2(\|x-y\|)| 
\leq C | \operatorname*{dist}(\sigma,\tau)|^2 \left(1 + |\log \operatorname*{dist}(\sigma,\tau)|\right)  e^{-b p}. 
$$
\end{lemma}
\begin{proof}
The proof follows with the tool developed in \cite{boerm-melenk17}. Consider 
$Q:= \prod_{i=1}^n [a_i,b_i] \subset {\mathbb R}^n$ and a function $f \in C(Q;{\mathbb C})$. Denote by $\Lambda_p$ 
the Lebesgue constant for univariate Chebyshev interpolation (note that $\Lambda_p = O(\log p)$). 
Introduce, for each $x \in Q$ and each $i \in \{1,\ldots,n\}$, the univariate function 
$f_{x,i}:[-1,1] \rightarrow {\mathbb C}$ by 
$f_{x,i}(t):= f(x_1,\ldots,x_{i-1}, (a_i+b_i)/2 + t(b_i-a_i)/2,x_{i+1},\ldots,x_n)$. 
Then, standard tensor product arguments 
\cite[Lemma~{3.3}]{boerm-melenk17} 
show that the tensor product Chebyshev interpolation error is bounded by 
$$
\|f - I^{Cheb}_p f\|_{L^\infty(Q)} \leq (1 + \Lambda_p) \Lambda_p^{n-1} \sum_{i=1}^n  
\sup_{x \in Q} \inf_{\pi \in {\mathbb P}_p} \|f_{x,i} - \pi\|_{L^\infty(-1,1)}. 
$$
The best approximation problems
$\inf_{\pi \in {\mathbb P}_p} \|f_{x,i} - \pi\|_{L^\infty(-1,1)}$ in turn lead to exponentially small (in $p$) 
errors, provided the holomorphic extensions of the functions $f_{x,i}$ can be controlled. We show this for the 
case $f(x_1,x_2,x_3,x_4) = \phi_2(\|(x_1,x_2) - (x_3,x_4)\|_2)$ under consideration here. Note that 
$f_{x,1}(t) = \phi_2(\|{\mathfrak d} - t{\mathfrak p}\|_2)$, where ${\mathfrak d} = ((a_1+b_1)/2-x_3, x_2-x_4)^\top$
and ${\mathfrak p} = ((a_1-b_1)/2,0)^\top$. Note 
$\|{\mathfrak d}\|_2 \leq (1+\eta) \operatorname*{dist}(\sigma,\tau)$ and 
$\|{\mathfrak p}\|_2 \leq 1/2 \max\{\operatorname*{diam}(\sigma),\operatorname*{diam}(\tau)\} 
\leq \eta/2 \operatorname*{dist}(\sigma,\tau)$.  
As is shown 
in 
\cite[Lemma~{3.6}, proof of Thm.~{3.13}]{boerm-melenk17}, 
the holomorphic extension of the function 
${\mathfrak n}: t \mapsto \|{\mathfrak d} - t {\mathfrak p}\|_2$ is holomorphic on $U_r:= \cup_{t \in [-1,1]} B_r(t)$
with $r = \operatorname*{dist}(\sigma,\tau)/\|{\mathfrak p}\|_2 \ge 2/\eta$ and maps into the left half plane 
${\mathbb C}_+ = \{z \in {\mathbb C}\,|\, \operatorname*{Re} z > 0\}$. We note that 
$\sup_{z \in U_r} |{\mathfrak n}(z)| \leq \|{\mathfrak d}\|_2 + r \|{\mathfrak p}\|_2
\leq (2+\eta) \operatorname*{dist}(\sigma,\tau)$. In view of $\phi_2(z) = z^2 \log z$, we conclude 
$\sup_{z \in U_r} |f_{x,i}(z)| \leq C 
(\operatorname*{dist}(\sigma,\tau))^2 (1 + 
|\log \operatorname*{dist}(\sigma,\tau)|)$ for a constant $C > 0$ that depends solely on $\eta$. We finish the 
proof by observing that there is $\rho > 1$ (depending only on $r$ and thus on $\eta$) 
such that $U_r$ contains the Bernstein ellipse ${\mathcal E}_\rho$ 
(see \cite[Lemma~{3.12}]{boerm-melenk17}).
A classical polynomial approximation result 
(see, e.g., \cite[Lemma~{3.11}]{boerm-melenk17}) 
concludes the proof. 
\end{proof}
\subsection{Edge effects and concentrating points at the boundary}
The convergence behavior of thin plate splines is limited by edge effects. Above, we mentioned that 
imposing certain boundary conditions on $f$ mitigates this effect. 
An alternative is to suitably concentrate 
points near $\partial\Omega$. Without proof, we announce the following result: 
\begin{proposition}
\label{prop:bdy-concentrated}
Assume that the points $x_i$, $i=1,\ldots,N$, satisfy for a $\delta > 0$ sufficiently small 
\begin{equation}
\forall x \in \Omega \colon \qquad 
\inf_{i=1,\ldots,N} \operatorname*{dist}(x,x_i) \leq 
\delta \min \left\{ h_{min} + \operatorname*{dist}(x,\partial\Omega) , h\right\}. 
\end{equation}
Then, for $f \in H^{m+1}(\Omega)$ there  holds $|f - If|_{H^m(\Omega)} \leq C \left( h_{min}^{1/2} + h\right)
|f|_{H^{m+1}(\Omega)}$. 
\end{proposition} 
Inserting the result of Proposition~\ref{prop:bdy-concentrated} in the estimates of Proposition~\ref{prop:duchon} 
shows that a factor $h_{min}^{1/2} + h$ can be gained in the convergence estimates. 
Fig.~\ref{fig:bdy-concentrated} presents numerical examples for the square $\Omega_1$ and the functions given 
in Sec.~\ref{sec:numerics}. We selected $h_{min} = h^2$ and distributed the points so as ensure the condition 
$$
\forall i \colon \qquad 
\min_{j\colon j\ne i} \|x_i - x_j\|_2 \gtrsim \min\left\{ h_{min} + \operatorname*{dist}(x,\partial\Omega), h\right\}. 
$$
For the present case $d = 2$, it can then be shown that the number of points $N$ is $O(h^{-2})$, which is 
also illustrated in Fig.~\ref{fig:bdy-concentrated}.
\begin{figure}
\includegraphics[width=0.49\textwidth]{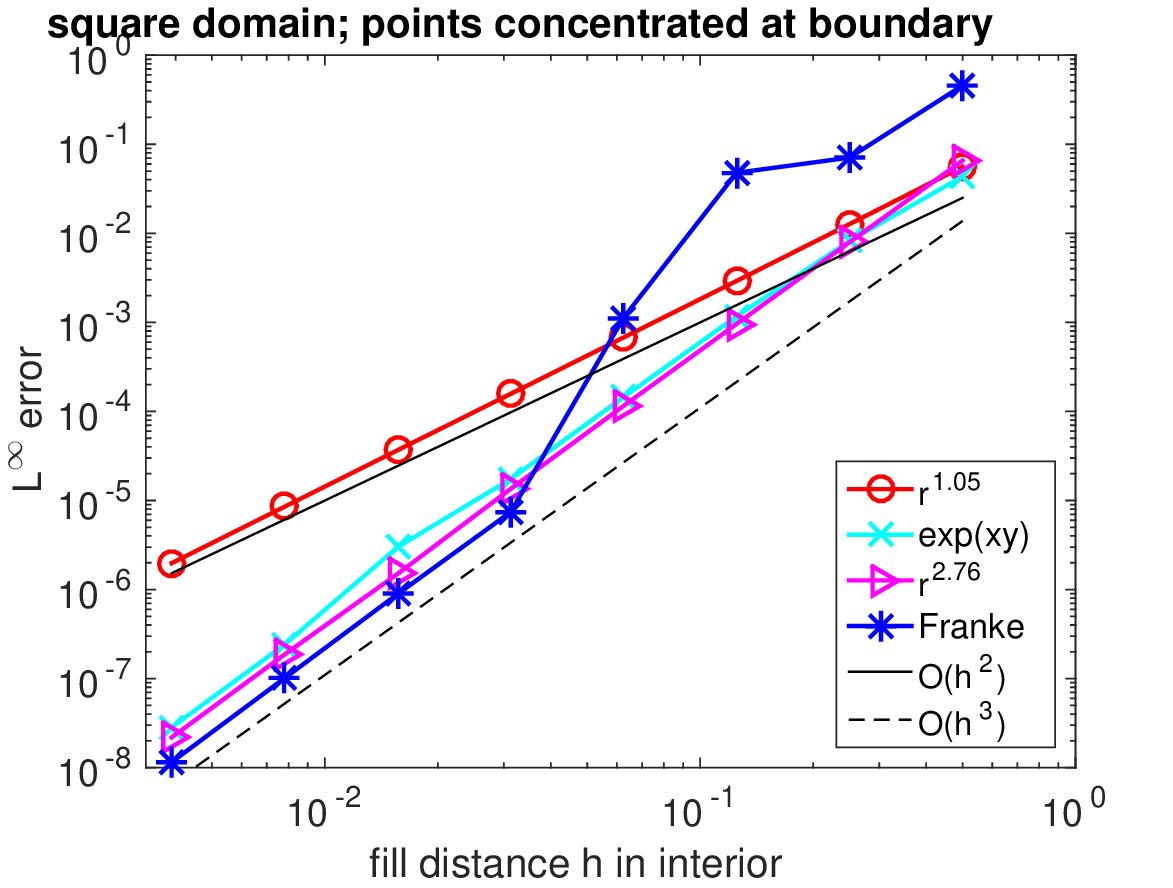}
\includegraphics[width=0.49\textwidth]{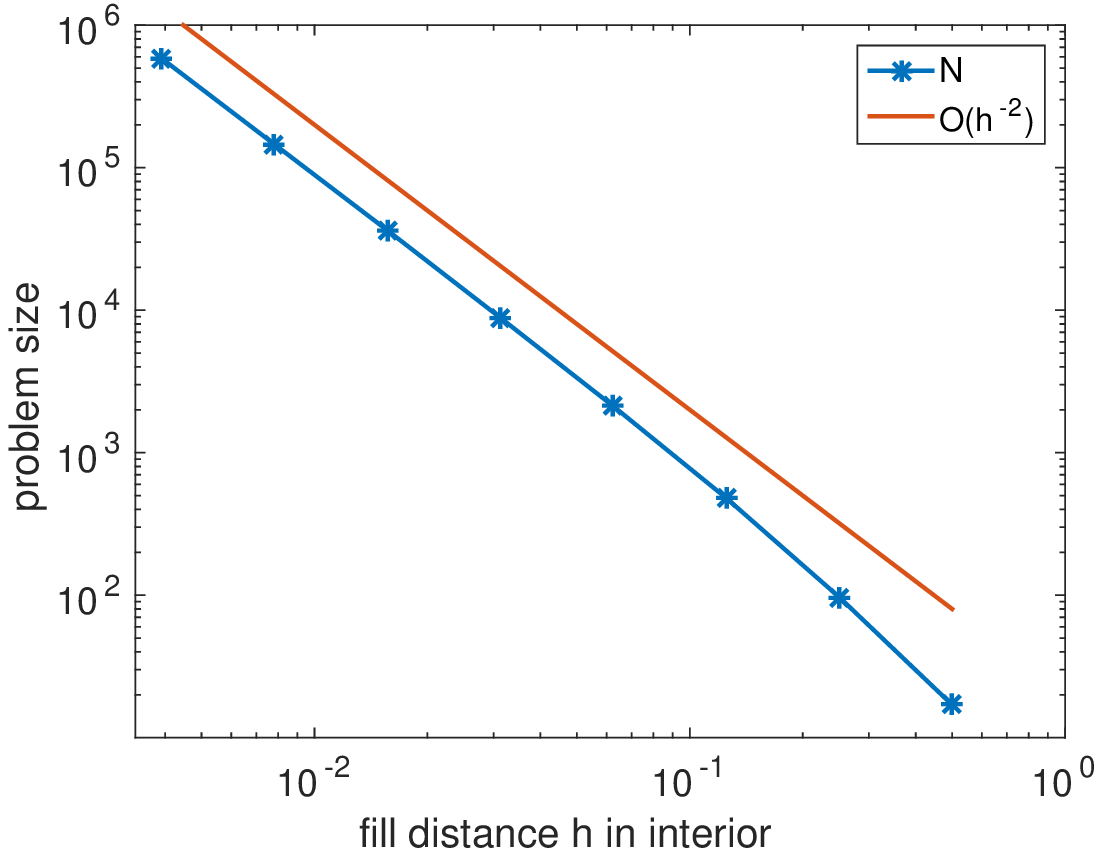}
\caption{\label{fig:bdy-concentrated} Concentrating points near $\partial\Omega$. Left: Convergence. 
Right: problem size versus fill distance in the interior.}
\end{figure}
\bibliography{nummech,eigene}
 \bibliographystyle{plain}
\end{document}